\documentclass[12pt]{amsart} 
\usepackage{amsmath,amssymb,amscd,amsthm,bm}
\usepackage{latexsym}
\usepackage{graphicx}
\usepackage[english]{babel}
\usepackage[latin1]{inputenc}       
\usepackage{textcomp}
\usepackage[margin=1.4in]{geometry}
\usepackage{times}
\usepackage{pgf,pgfarrows,pgfnodes,pgfautomata,pgfheaps}
\usepackage{colortbl}

\newtheorem{theorem}{Theorem}

\newtheorem{lemma}{Lemma}

\newcommand{\sign}{\mathrm{sign}}

\begin{document}

\title[Poincar\'e 3/2 ]{Poincar\'e inequality 3/2 on the Hamming cube}
\author[Paata Ivanisvili and  Alexander Volberg]{Paata Ivanisvili and Alexander Volberg}
\address{Department of Mathematics, Kent State University, Kent, OH 44240}
\email{ivanishvili.paata@gmail.com}
\address{Department of Mathematics, Michigan State University} 
\email{volberg@math.msu.edu}
\thanks{AV is partially supported by the NSF grant DMS-1600065 and by the Hausdorff Institute for Mathematics, Bonn, Germany}

\begin{abstract} For any $n \geq 1$, and any $f :\{-1,1\}^{n} \to \mathbb{R}$ we have 
\begin{align*}
\Re\, \mathbb{E}\,  (f + i\,  |\nabla f|)^{3/2} \leq \Re\,  (\mathbb{E}f)^{3/2},
\end{align*}
where  $z^{3/2}$ for $z=x+iy$ is taken with principal branch and  $\Re$ denotes the real part. 
\end{abstract}
\maketitle

\section{A peculiar function 3/2}
Fix any integer $n \geq 1$ and consider the Hamming cube $\{-1,1\}^{n}$ equipped with the uniform counting measure $d\mu$.   Let $f : \{-1, 1\}^{n} \to\mathbb{R}$ be an arbitrary function. Define the \emph{directional derivative} at point $x=(x_{1}, x_{2}, \ldots, x_{n}) \in \{-1,1\}^{n}$ as follows
\begin{align*}
\nabla_{j} f(x) \stackrel{\mathrm{def}}{=}  \frac{1}{2} \left(f\underbrace{(x_{1},x_{2}, \ldots, 1, \ldots, x_{n})}_{\text{set $1$ on $j$-th place}} -f\underbrace{(x_{1},x_{2}, \ldots, -1, \ldots, x_{n})}_{\text{set $-1$ on $j$-th place}}  \right).
\end{align*}
Next we define the \emph{gradient} as $\nabla f = (\nabla_{1} f, \ldots, \nabla_{n} f)$ and set
\begin{align*}
|\nabla f (x)|^{2}  \stackrel{\mathrm{def}}{=} \sum_{j=1}^{n} |\nabla_{j} f(x)|^{2}\quad \text{for all} \quad  x \in \{-1,1\}^{n}.
\end{align*}
Our goal is to prove the following theorem
\begin{theorem}
For any $n \geq 1$, and any $f :\{-1,1\}^{n} \to \mathbb{R}$ we have 
\begin{align}\label{mth1}
\Re\, \mathbb{E}\,  (f + i\,  |\nabla f|)^{3/2} \leq \Re\,  (\mathbb{E}f)^{3/2},
\end{align}
where  $z^{3/2}$ for $z=x+iy$ is taken with principal branch and  $\Re$ denotes the real part. 
\end{theorem} 
 
\section{Proof of the theorem}
Let $z=x+iy$ where $x,y \in \mathbb{R}$. For $\arg(z) \in  (-\pi, \pi]$ define  $M(x,y)=\Re\,  z^{3/2}$.  Notice that 

\begin{align*}
M(x,y)=\frac{1}{\sqrt{2}}(2x-\sqrt{x^{2}+y^{2}}) \sqrt{\sqrt{x^{2}+y^{2}}+x}.
\end{align*}
Inequality (\ref{mth1}) takes the form 
\begin{align}\label{mth2}
\mathbb{E} \, M(f, |\nabla f|) \leq M(\mathbb{E}\, f, 0). 
\end{align}
The proof of (\ref{mth2}) essentially is based on the following \emph{main inequality}
\begin{lemma}\label{mlemma}
For any $x,y,a,b \in \mathbb{R}$ we have 
\begin{align}\label{lemmain}
M(x,y) \geq \frac{1}{2}\left(M(x+a, \sqrt{a^{2}+(y+b)^{2}}) + M(x-a,\sqrt{a^{2}+(y-b)^{2}}) \right).
\end{align}
\end{lemma}
Before we proceed to the proof of Lemma~\ref{mlemma} let us explain that (\ref{lemmain}) implies (\ref{mth2}).  First we notice that Lemma~\ref{mlemma} implies a stronger  inequality 
\begin{align}\label{stronger1}
M(x,\| y\|)\geq \frac{1}{2}\left(M(x+a, \sqrt{a^{2}+\| y+b\|^{2}}) +M(x-a, \sqrt{a^{2}+\| y-b\|^{2}})  \right)
\end{align}
for all $x,a \in \mathbb{R}$, all $y,b \in \mathbb{R}^{N}$ and  any $N \geq 1$.  Indeed, by Lemma~\ref{mlemma} we have 
\begin{align*}
&\frac{1}{2}\left(M(x+a, \sqrt{a^{2}+\| y+b\|^{2}}) +M(x-a, \sqrt{a^{2}+\| y-b\|^{2}})  \right) \leq \\
&M\left( x, \frac{\| y+b\| + \| y-b\|}{2}\right) \leq M\left( x, \| y\|\right).
\end{align*}
The last inequality follows from the fact that $\frac{\| y+b\| + \| y-b\|}{2} \geq  \| y \|$ and the map $t \to M(x,t)$ is decreasing for $t \geq 0$:
\begin{align*}
\frac{\partial }{\partial t}M(x,t) = -\frac{3}{2\sqrt{2}} \sqrt{\sqrt{x^{2}+t^{2}}-x}<0, \quad \text{for} \quad t \geq 0.
\end{align*}

It will be convenient for us to use martingale notation but of course one can proceed without invoking these notations. Define the martingale $\{ f_{k}\}_{k=0}^{n}$ as follows: let $f_{k} = \mathbb{E} (f | \mathcal{F}_{k})$ to be the average of the function $f$ with respect to the variables $(x_{k+1}, \ldots, x_{n})$. For example 
\begin{align*}
&f_{n} =f; \\
&f_{n-1} = \frac{1}{2}\left( f(x_{1}, \ldots, x_{n-1},1)+f(x_{1}, \ldots, x_{n-1},-1)\right);\\
&\ldots  \\
&f_{0} = \frac{1}{2^{n}} \sum_{x \in \{-1,1\}^{n}} f(x) = \mathbb{E} f.
\end{align*}
Thus $f_{k}$ lives on $\{-1,1\}^{k}$ for $1 \leq k \leq n$. 

Next we would like to know how the next generation $k+1$ is related to the previous generation $k$. 
For $x \in \{-1,1\}^{k+1}$ let $x=(x',x_{k+1})$ where $x' \in \{-1,1\}^{k}$.   Notice that
\begin{align*}
&f_{k+1}(x', x_{k+1}) = f_{k}(x') +x_{k+1} \cdot g(x');\\
&|\nabla f_{k+1}(x',x_{k+1})|^{2} =  | \nabla_{x'} (f_{k}(x')+x_{k+1}\cdot g(x'))|^{2}+|g(x')|^{2}.
\end{align*}
where $g=g^{k}$ is a function on $\{-1,1\}^{k}$, and $\nabla_{x'}$ denotes gradient taken in $x'$.

 We claim that the following process 
\begin{align*}
z_{k} = M(f_{k}, |\nabla f_{k}|), \quad 0 \leq k \leq n
\end{align*}
is a supermartingale after which the theorem follows immediately:
\begin{align*}
M(\mathbb{E} f, 0) = z_{0} \geq  \mathbb{E} z_{n} = \mathbb{E} M(f, |\nabla f|).
\end{align*}

To verify the claim we notice that  
\begin{align*}
&\mathbb{E} (z_{k+1}| \mathcal{F}_{k})(x') = \frac{1}{2}\left(z_{k+1}(x',1)+z_{k+1}(x',-1) \right)=\\
&\frac{1}{2}\left( M(f_{k}(x')+g(x'),\sqrt{| \nabla_{x'} (f_{k}(x')+ g(x'))|^{2}+|g(x')|^{2}} )+  \right.\\
&\left. M(f_{k}(x')-g(x'),\sqrt{| \nabla_{x'} (f_{k}(x')-g(x'))|^{2}+|g(x')|^{2}} ) \right) \leq \\
&M(f_{k}(x'), |\nabla f_{k}(x')|)=z_{k}.
\end{align*}
The last inequality follows from (\ref{stronger1}) where we set $x=f_{k}(x'), a=g(x'), y = \nabla_{x'} f_{k}(x')$ and $b=\nabla_{x'} g(x')$.

\section{Proof of Lemma~\ref{mlemma}}
  Brownian motion approach  developed in  \cite{Maurey11} can be used to obtain (\ref{lemmain}) but only  for a particular case $b=0$. The general case $ b \neq 0$ is essential for our purposes and it creates difficulty in proving  (\ref{lemmain}).  The proof we are going to present is straightforward and can be checked by hand. However, it is difficult to imagine  how to come up with these computations and identities without using a computer. Our proof is computer assisted. 

Without loss of generality we can make the following assumptions: 1) $y \geq 0$; 2) $|b| \leq y$ (since $M$ is monotone in $y$); 3) $a \neq 0$ (otherwise inequality follows from concavity of the map $y \to M(x,y)$); 4) $ a>0$ (change sign of $b$ if necessary). 

Consider the function 
\begin{align*}
E(t) \stackrel{\mathrm{def}}{=} M(x+at, \sqrt{(at)^{2}+(y+bt)^{2}})+M(x-at, \sqrt{(at)^{2}+(y-bt)^{2}}), \quad t \in [0,1].
\end{align*}
It is enough to show that $E(t)$ is decreasing for $t \in [0,1]$.  We will  set $a=1$ and consider $E(t)$ on the interval $[0,a]$ (but now $b \to b/a$). 
We have 

\begin{align*}
&M_{x} = \frac{3}{2\sqrt{2}} \sqrt{\sqrt{x^{2}+y^{2}}+x};\\
&M_{y} = -\frac{3}{2\sqrt{2}} \sqrt{\sqrt{x^{2}+y^{2}}-x}, \quad y \geq 0. 
\end{align*}

Notice that if we use the fact $M_{x} M_{y} = -\frac{9}{8} y$ we obtain:
\begin{align*}
&E'(t) = M^{+}_{x}+M_{y}^{+} \frac{t+b(y+bt)}{\sqrt{t^{2}+(y+bt)^{2}}}- M_{x}^{-} + M_{y}^{-} \frac{t-b(y-bt)}{\sqrt{t^{2}+(y-bt)^{2}}} =\\
&\frac{9}{8 M^{+}_{x}}\left[(x+t)+ \sqrt{(x+t)^{2}+t^{2}+(y+bt)^{2}}  - (t+b(y+bt)) \right]\\
&-\frac{9}{8 M^{-}_{x}}\left[(x-t)+ \sqrt{(x-t)^{2}+t^{2}+(y-bt)^{2}}  +(t-b(y-bt)) \right]
\end{align*}
Where $M^{+}$ and $M^{-}$ are computed at the points $(x+t, \sqrt{t^{2}+(y+bt)^{2}})$ and $(x-t, \sqrt{t^{2}+(y-bt)^{2}})$ correspondingly.

Next we can always assume (by homogeneity $M(\lambda x, \lambda y)=\lambda^{3/2}M(x,y)$ and considering new variables $\tilde{x} = xt$, $\tilde{y} = yt$) that $t=1$. Thus we need to show that 
\begin{align}\label{main}
\frac{x-by-b^{2}+ \sqrt{(x+1)^{2}+1+(y+b)^{2}}}{\sqrt{x+1+\sqrt{(x+1)^{2}+1+(y+b)^{2}}}} \leq \frac{x-by+b^{2}+ \sqrt{(x-1)^{2}+1+(y-b)^{2}}}{\sqrt{x-1+\sqrt{(x-1)^{2}+1+(y-b)^{2}}}}
\end{align}
and $|b| \leq y$. 


Consider the difference: the left hand side of (\ref{main}) minus the right hand side of (\ref{main})  as a function of $x$, and call it $f(x)$.
 We want to show that $f(x) \leq 0$. The function also depends on $b,y$, in fact $f(x)=f(x,b,y)$ is real analytic in $x,b,y$. 
 
\begin{lemma}
We have 
\begin{align*}
&f(x) =- b^{2} \sqrt{2} \cdot x^{-1/2}+O(x^{-3/2}) \quad \text{as} \quad x \to \infty;\\
&f(x) =-\frac{\sqrt{-2x}  \left((1+b^{2}+by)\sqrt{1+(y-b)^{2}}+(1+b^{2}-by)\sqrt{1+(y+b)^{2}} \right)}{\sqrt{(1+(y+b)^{2})(1+(y-b)^{2})}}+O((-x)^{-1/2})\\
& \text{as} \quad x \to -\infty;
\end{align*}
And the signs of $f(x)$ are  negative at $\pm \infty$. 
\end{lemma}
\begin{proof}
The proof is pretty straightforward. Case $b=0$ will be mentioned  later.
\end{proof}

Let us try to find possible roots of $f(x)$.


After squaring  (\ref{main}) and simplifying the expressions we end up with the following equation
\begin{align}\label{main1}
C_{A}\cdot A + C_{B} \cdot B + C_{AB} \cdot A \cdot B + L =0
\end{align}
where 
\begin{align*}
&C_{A}=4by - 4b^{2}x + b^{2} - b^{2}y^{2} + 2b^{3} y - b^{4} -2- y^{2};\\
&C_{B}= -4b^{2} x + b^{2}y^{2} + 2b^{3}y + b^{4} +2 + y^{2} + 4by -b^{2};\\
&C_{AB}=-4b^{2};\\
&L=-4-4b^{2}x^{2}+4b^{3}yx-2b^{4}+8byx-2b^{2}-2b^{2}y^{2}-2y^{2};\\
&A= \sqrt{(x+1)^{2}+1+(y+b)^{2}};\\
&B = \sqrt{(x-1)^{2}+1+(y-b)^{2}}. 
\end{align*}

After moving terms $L, C_{AB} \cdot A \cdot B$ to the right hand side of the equation, squaring and moving some terms again, and squaring again  we finally obtain that 
\begin{align*}
(C_{A}^{2}\cdot A^{2}+C_{B}^{2}\cdot B^{2}-L^{2}-C_{AB}^{2}\cdot A^{2}\cdot B^{2})^{2}-4\cdot A^{2}\cdot B^{2}\cdot (C_{AB}\cdot L-C_{A}\cdot C_{B})^{2}=0
\end{align*}
Lets denote the left hand side of the equation by $P(x)$. This is a 3rd degree polynomial in $x$. We have
\begin{align*}
&P(x) =\\
&-128b^{3}y^{3}(b^{2}y^{2}+y^{2}+2+4by+3b^{2}+2b^{3}y+b^{4})(b^{2}y^{2}+y^{2}+2-4by+3b^{2}-2b^{3}y+b^{4})x^{3}\\
&+(-64y^{8}b^{8}+1088b^{6}y^{6}-3392b^{8}y^{4}+8128b^{10}y^{2}+384b^{10}y^{6}-704b^{12}y^{4}+960b^{8}y^{6}-3136b^{10}y^{4}\\
&+3392b^{12}y^{2}+512b^{14}y^{2}-64y^{8}b^{6}+64y^{8}b^{4}+64y^{8}b^{2}-960b^{4}y^{6}+960b^{6}y^{4}+64b^{2}y^{6}\\
&-2816b^{4}y^{2}+1280b^{4}y^{4}+1088b^{6}y^{2}-640b^{2}y^{4}+7872b^{8}y^{2}-1280b^{2}y^{2}-10880b^{8}\\
&-8960b^{10}-3072b^{4}-128b^{16}-7808b^{6}-512b^{2}-4352b^{12}-1152b^{14})x^{2}\\
&+(-1792b^{5}y^3+256b^7y^7-5504b^7y^3-1408b^5y^7+3456b^7y^5-384y^7b^3+640b^9y^5\\
&+2752b^5y^5+1536b^3y^3-5760b^9y^3-3840b^{11}y^3-768b^3y^5+512by+3072b^3y\\
&+1024by^3+1984b^{13}y+384b^{15}y+32b^{17}y+32by^9+10272b^9y+768by^5\\
&+5760b^{11}y+256by^7+32b^9y^9-128b^{11}y^7-1408b^{13}y^3-64b^5y^9\\
&-640b^9y^7+1664b^{11}y^5+192b^{13}y^5-128b^{15}y^3+7936b^5y+11520b^7y)x
\end{align*}
\begin{align*}
&-256-144b^{18}-16y^{10}+688y^8b^8+1504b^6y^6-1920b^8y^4-3440b^{10}y^2\\
&-2304b^{10}y^6+2592b^{12}y^4-192b^8y^6+3264b^{10}y^4-4448b^{12}y^2-352b^{14}y^{2}\\
&-288y^8b^6-224y^8b^4+48y^8b^2-736b^4y^6-1376b^6y^4-320b^2y^6-2816b^4y^2\\
&-480b^4y^4+2496b^6y^2-1792b^2y^4+3056b^8y^2-3072b^2y^2-768y^2-512y^6-896y^4\\
&-144y^8-3344b^8+1584b^{10}-4992b^4-336b^{16}-6656b^6-1792b^2+2528b^{12}\\
&+608b^{14}-64b^{16}y^4+96b^{14}y^6+16y^{10}b^2+32y^{10}b^4+624b^{16}y^2-864b^{14}y^4\\
&+416b^{12}y^6-64b^{12}y^8-16b^{10}y^8-16b^8y^{10}+16b^{10}y^{10}-32y^{10}b^6+16b^{18}y^2
\end{align*}
The discriminant of this polynomial  turns out to factorize as follows:
\begin{align*}
&\Delta = 16777216\cdot (1+b^{2})^{2}\cdot (-8-16b^{2}-8b^{4}-8y^{2}+20b^{2}y^{2}+b^{4}y^{2}-2y^{4}-2b^{2}y^{4})\cdot \\
&(-b^{4}y^{2}+2b^{2}y^{2}-y^{2}-2-3b^{2}+b^{6})^{2}((b^{2}y^{2}+y^{2}+2+3b^{2}+b^{4})^{2}-(4by+2b^{3}y)^{2})^{2}\cdot \\
&(4+24b^{2}+3b^{12}+76b^{6}+54b^{8}+20b^{10}+4y^{8}+14y^{6}+17y^{4}+12y^{2}+59b^{4}-14b^{6}y^{6}\\
&+19b^{8}y^{4}-12b^{10}y^{2}+4y^{8}b^{4}+8y^{8}b^{2}-22b^{4}y^{6}+46b^{6}y^{4}+6b^{2}y^{6}+4b^{4}y^{2}+20b^{4}y^{4}\\
&-52b^{6}y^{2}+26b^{2}y^{4}-48b^{8}y^{2}+32b^{2}y^{2})^{2}\cdot b^{6}=\\
&16777216\cdot (1+b^{2})^{2}\cdot  T_{1} \cdot  T_{2}^{2} \cdot T_{3}^{2}  \cdot T_{4}^{2} \cdot b^{6}.
\end{align*}

If $b=0$ then 
\begin{align*}
P(x) = -16 (y^{2}+1)(y^{2}+2)^{4}<0.
\end{align*}
This means that $f(x)$ does not have roots (in particular this proves (\ref{lemmain}) because $f(-\infty)<0$). Therefore further we assume that $ b\neq 0$. Next if $y=0$ then 
\begin{align*}
P(x)=-16(b^{2}+1)^{5}(8b^{2}(b^{2}+2)^{2}x^{2}+(3b^{2}+2)^{2}(b^{2}-2)^{2})<0,
\end{align*}
Which again means that $f(x)$ does not have roots and hence $f(x)<0$ in this case as well.
 Next we  assume that $b, y \neq 0$. 
 
 We should investigate the sign of the  discriminant $\Delta$.  First consider the longest term $T_{4}$  in the discriminant. 
\begin{lemma}
We have $T_{4}>0$, i.e., 
\begin{align*}
&4(1+b^{2})^{2}y^{8}-2(1+b^{2})(7b^{4}+4b^{2}-7)y^{6}+(19b^{8}+26b^{2}+20b^{4}+46b^{6}+17)y^{4}\\
&-4(3b^{6}+6b^{4}-2b^{2}-3)(1+b^{2})^{2}y^{2}+(3b^{2}+2)(b^{2}+2)(1+b^{2})^{4}>0. 
\end{align*}
\end{lemma}
\begin{proof}The proof is direct application of Sturm's theorem.
Consider the polynomial
\begin{align*}
&g(y):= 4(1+b^{2})^{2}y^{4}-2(1+b^{2})(7b^{4}+4b^{2}-7)y^{3}+(19b^{8}+26b^{2}+20b^{4}+46b^{6}+17)y^{2}\\
&-4(3b^{6}+6b^{4}-2b^{2}-3)(1+b^{2})^{2}y+(3b^{2}+2)(b^{2}+2)(1+b^{2})^{4}
\end{align*}
for $y \geq 0$. 
Lets compute the Sturm's sequence for it.
We obtain $g_{0}=g, g_{1}=g'(y), g_{2}, g_{3}, g_{4}$.  We would like to show that $g$ does not  have roots on $[0,\infty)$. So we have two vectors of signs at points $0$ and $\infty$: 
\begin{align*}
&\sign (g(0), g_{1}(0), g_{2}(0), g_{3}(0), g_{4}(0))=\sign (+, g_{1}(0),-, g_{3}(0), g_{4}(0))=u(b)\\
&\sign(g(\infty), g_{1}(\infty), g_{2}(\infty), g_{3}(\infty), g_{4}(\infty)) = \sign(+, +, g^{(2)}_{2}(0), g^{(1)}_{3}(0), g_{4}(0))=v(b)
\end{align*} 
where 
\begin{align*}
&g(0)=(3b^{2}+2)(b^{2}+2)(b^2+1)^{4}>0;\\
&g_{1}(0) = -4(3b^6+6b^4-2b^2-3)\cdot (b^2+1)^2;\\
&g_{2}(0) = -\frac{1}{8}(b^2+1)(3b^{10}+82b^8+307b^6+383b^4+158b^2+11)<0;\\
&g_{3}(0) = 32(b^{2}+1)^{2}(3b^{22}-37b^{20}-928b^{18}-74b^{16}+8954b^{14}-4262b^{12}\\
&-35980b^{10}+12864b^8+54811b^6+25171b^4+1044b^2-638)\cdot \\
&(5b^8+200b^6+406b^4+376b^2-11)^{-2};\\
&g_{4}(0)=\frac{1}{16}(-59+1060b^2+2182b^4+560b^6-91b^8+12b^{10})(b^8+b^6-13b^4+11b^2+8)^{2}\cdot \\
&(5b^8+200b^6+406b^4+376b^2-11)^{2}(b^2+1)^{8}\cdot (3b^{24}-108b^{22}-978b^{20}+3700b^{18}+16069b^{16}\\
&-36120b^{14}-78876b^{12}+96712b^{10}+112317b^8-34812b^6-47410b^4-6844b^2+923)^{-2};
\end{align*}
and 
\begin{align*}
&\sign(g(\infty))=\sign(4(b^2+1)^2)=+;\\
&\sign(g_{1}(\infty)) = \sign(16(b^{2}+1)^{2})=+;\\
&\sign(g_{2}(\infty)) = \sign\left( -\frac{5}{8}b^8-25b^6-\frac{203}{4}b^4-47b^2+\frac{11}{8}\right);\\
&\sign(g_{3}(\infty)) = \sign(-32(3b^{24}-108b^{22}-978b^{20}+3700b^{18}+\\
&16069b^{16}-36120b^{14}-78876b^{12}+96712b^{10}+112317b^8-34812b^6\\
&-47410b^4-6844b^2+923);\\
&\sign(g_{4}(\infty))  = \sign(g_{4}(0));
\end{align*}
These vectors of signs  $u(b)$ and $v(b)$ depend on $b$. 
It requires computations   to show that they have the same number of sign changes. Indeed, consider $g_{4}(0)$.  Its sign coincides with sign of 
\begin{align*}
-59+1060b^2+2182b^4+560b^6-91b^8+12b^{10}.
\end{align*}
It  has only one real root  on interval $[0, \infty)$ (again by Sturm's theorem). Applying Sturm's theorem one more time but now  to the interval $[0.22, 0.23]$ one obtains that the positive root $b_{0}$  belongs to this interval. Since $g_{4}(0)<0$ when $b=0$ one obtains that $g_{4}(0)<0$ when $0<|b|<b_{0}$ and $g_{4}(0)>0$ when $|b|>b_{0}$. 

Next consider two cases. First case when $0<|b|<b_{0}$. In that case  it is enough  to show that $g_{3}^{(1)}(0)<0$  and $g_{3}(0)<0$ which is true. Indeed, applying Sturm's theorem to the interval $[0,0.23]$ we see that polynomials $g_{3}^{(1)}(0)$ and $g_{3}(0)$ do not have real roots on $[0, 0.23]$. On the other hand when $b=0$ we have $g_{3}^{(1)}(0)<0$ and $g_{3}(0)<0$ which implies that these polynomials are negative on $[0,0.23]\supset[0,b_{0}]$ and the conclusion follows. 

Next consider another case when $|b|>b_{0}$. In that case since $g_{4}(0)>0$ it will be enough to show that $g_{2}^{(2)}(0)<0$. Applying Sturm's theorem to the interval $[0.22, \infty)$ we see that $g_{2}^{(2)}(0)$ does not have roots. On the other hand when $b \to \infty$ we see that $g_{2}^{(2)}(0)<0$ and therefore $g_{2}^{(2)}(0)<0$ on $[0.22, \infty) \supset [b_{0}, \infty)$. 

Thus we have proved that the long expression  $T_{4}$ in the discriminant never becomes zero. 
\end{proof}

Next we show that $T_{3}>0$. 
\begin{lemma}
We have 
\begin{align*}
&T_{4}=(b^{2}y^{2}+y^{2}+2+3b^{2}+b^{4})^{2}-(4by+2b^{3}y)^{2}>0.
\end{align*}
\end{lemma}
\begin{proof}

It is enough to show that $b^{2}y^{2}+y^{2}+2+3b^{2}+b^{4}\pm (4by+2b^{3}y) >0$.
We can consider only one expression (because the other one is just substitution $b\to -b$ but  we are proving the lemma for all $b$). 
Take $r(y):=b^{2}y^{2}+y^{2}+2-4by+3b^{2}-2b^{3}y+b^{4}$. This is a parabola going upward. Its value on its local minimum is 
\begin{align*}
r\left(-\frac{b(b^{2}+2)}{b^{2}+1}\right) = \frac{2+b^{2}}{1+b^{2}} >0.
\end{align*}
\end{proof}

We left with  the first two nontrivial terms. Lets start from $T_{2}$. 

\begin{align*}
w(y):=-b^{4}y^{2}+2b^{2}y^{2}-y^{2}-2-3b^{2}+b^{6}.
\end{align*}
This expression becomes zero only when 
\begin{align*}
y^{2} = \frac{(b^{2}-2)(b^{2}+1)^{2}}{(b^{2}-1)^{2}}.
\end{align*}

We should exclude cases $b^{2}=1$ because in that cases $w(y)=-4$ so it is never zero.  If $b^{2}=2$ then $y=0$ and in that case we have proved the inequality. So we left that $b^{2}>2$. Thus we have that 
\begin{align*}
y = \frac{(b^{2}+1)\sqrt{b^{2}-2}}{b^{2}-1};
\end{align*}

In this case $P(x)$ has a root of multiplicity $2$ which turns out to be $x=b\sqrt{b^{2}-2}$ . 
We just need to make sure that at this root $f(x)$ is not zero (we can acquire zeros of course by going from $f$ to $P$). Then $f(x)$ may have at most 1 root but since it has negative signs at $\pm \infty$ we are done. 
So assuming $y = \frac{(b^{2}+1)\sqrt{b^{2}-2}}{b^{2}-1}$ and $x=b\sqrt{b^{2}-2}$ we obtain  that in the left hand side  of  (\ref{main}) we have 
\begin{align*}
&x-by-b^{2}+\sqrt{(x+1)^{2}+1+(y+b)^{2}}=\\
&-\frac{b(2\sqrt{b^{2}-2}+b^{3}-b)}{b^{2}-1}+ \sqrt{ \frac{b^{2}(2\sqrt{b^{2}-2}+b^{3}-b)^{2}}{(b^{2}-1)^{2}}}=0. 
\end{align*}

On the other hand lets see what is the right hand side of (\ref{main}):
\begin{align*}
&x-by+b^{2}+\sqrt{(x-1)^{2}+1+(y-b)^{2}}=\\
&-\frac{b(2\sqrt{b^{2}-2}-b^{3}+b)}{b^{2}-1}+\sqrt{\frac{b^{2}(2\sqrt{b^{2}-2}-b^{3}+b)^{2}}{(b^{2}-1)^{2}}}=\\
&-2\cdot \frac{b(2\sqrt{b^{2}-2}-b^{3}+b)}{b^{2}-1} >0 \quad  \text{for} \quad |b| \geq \sqrt{2}
\end{align*}

Thus we left with  $T_{1}$.

\begin{lemma}
If $|b| \leq y$ then 
\begin{align*}
T_{1}=-8-16b^{2}-8b^{4}-8y^{2}+20b^{2}y^{2}+b^{4}y^{2}-2y^{4}-2b^{2}y^{4}<0
\end{align*}
\end{lemma}
\begin{proof}
We have 

\begin{align*}
T_{1}(y)=(-2-2b^{2})y^{4}+(-8+20b^{2}+b^{4})y^{2}-8(1+b^{2})^{2}
\end{align*}
If $-8+20b^{2}+b^{4} \leq 0$ i.e.,  $|b| \leq \sqrt{-10+6\sqrt{3}}$ then we are done because it means that $T_{1}(0)<0$ and $T_{1}'(y) \leq T_{1}'(0) \leq 0$ so the discriminant  $\Delta$ of cubic polynomial $P(x)$ is negative. 

Next assume that $|b| > \sqrt{-10+6\sqrt{3}} >0$. The equation $T_{1}'(y)=0$ has a solution for $y  \geq 0$ and $y_{0} = \sqrt{\frac{1}{4} \frac{-8+20b^{2}+b^{4}}{1+b^{2}}}$. At this point $T_{1}(y)$ attains its maximal value which is 
\begin{align*}
T_{1}(y_{0}) = \frac{1}{8} \frac{b^{2}(b^{2}-8)}{1+b^{2}}
\end{align*}
and the maximal value is still negative for a while, namely if $b^{2}<8$.  Thus we are still fine  even  if $\sqrt{-10+6\sqrt{3}}<|b| < \sqrt{8}$. 

Otherwise, if $|b|\geq \sqrt{8} = 2\sqrt{2}$ we have two positive roots which means that $T_{1}$ has positive sign only if 
\begin{align}\label{interval1}
y \in \left[\sqrt{\frac{1}{4} \cdot \frac{-8+20b^{2}+b^{4}-\sqrt{b^{2}\cdot(b^{2}-8)^{3}}}{b^{2}+1}}, \sqrt{\frac{1}{4} \cdot \frac{-8+20b^{2}+b^{4}+\sqrt{b^{2}\cdot(b^{2}-8)^{3}}}{b^{2}+1}} \right]
\end{align}
and $|b| \geq \sqrt{8}$. 
But one can show that  since $y \geq |b|$ by the general assumption, besides  we also have the inequality 
\begin{align*}
|b| \geq \sqrt{\frac{1}{4} \cdot \frac{-8+20b^{2}+b^{4}+\sqrt{b^{2}\cdot(b^{2}-8)^{3}}}{b^{2}+1}} \quad \text{for} \quad |b| \geq \sqrt{8}
\end{align*}
which means that $y$ cannot belong to that interval (\ref{interval1}), so $T_{1}<0$. 
\end{proof}
Finally negativity of the discriminant $\Delta$ for $P(x)$ just means that $f(x)$ has at most one real root, but the fact that $f$ has negative signs at $\pm \infty$ implies $f\leq 0$. This finishes the proof.


\section{Applications} 
Beckner--Soblev inequality obtained by W.~Beckner in 1988  (see \cite{WB}) says that  for any smooth bounded $f \geq 0$ we have 
\begin{align}\label{bbeckner}
\int_{\mathbb{R}^{n}} f^{p} d\gamma  - \left( \int_{\mathbb{R}^{n}} f d\gamma\right)^{p} \leq \int_{\mathbb{R}^{n}} \frac{p(p-1)}{2} f^{p-2} |\nabla f|^{2} d\gamma, \quad p \in [1,2],
\end{align}
where $d\gamma = \frac{e^{-x^{2}/2}}{\sqrt{2 \pi}}dx$ is the standard $n$-dimensional Gaussian measure. The constant $\frac{p(p-1)}{2}$  in the right hand side of inequality (\ref{bbeckner}) is sharp as one can see on the example of test function $f(x) = e^{\varepsilon x}$ by sending $\varepsilon \to 0$ (here $n=1$).  Beckner--Sobolev inequality (\ref{bbeckner}) interpolates in a sharp way log-Sobolev inequality and Poincar\'e inequality \cite{PIVO11}. Inequality (\ref{bbeckner})  was studied in different settings, for different measures and in different spaces as well. 
For possible references we refer the reader to \cite{ABD, ALS, BGL, BCR1, BCR2, BR1, WB, Bob1, Bob2, BBL, Chaf, IV, KO, RK, WFY}. 

 Recently the authors \cite{PIVO11} improved (\ref{bbeckner}) essentially. Namely if $p=3/2$ we obtained that for any smooth bounded $f \geq 0$ we have 
\begin{align}
&\int_{\mathbb{R}^{n}} f^{3/2} d\gamma - \left(\int_{\mathbb{R}^{n}} f d\gamma \right)^{3/2} \leq  \label{RHS1} \\
&\int_{\mathbb{R}^{n}}\left(f^{3/2} -\frac{1}{\sqrt{2}}(2f-\sqrt{f^{2}+|\nabla f|^{2}})\sqrt{\sqrt{f^{2}+|\nabla f|^{2}}+f^{2}}\right) d\gamma.\nonumber
\end{align}

Integrand in the right hand side of (\ref{RHS1}) is strictly smaller than the integrand in the right hand side of (\ref{bbeckner}) for $p=3/2$. Indeed, notice that 
\begin{align}\label{impr2}
\left(x^{3/2} -\frac{1}{\sqrt{2}}(2x-\sqrt{x^{2}+y^{2}})\sqrt{\sqrt{x^{2}+y^{2}}+x^{2}}\right) \leq  \frac{3}{8}x^{-1/2}y^{2}, \quad x, y \geq 0,
\end{align}
which follows from the homogeneity (we can consider $x=1$). Next plugging $x=f$ and $y=|\nabla f|$ in (\ref{impr2}) and integrating we obtain that (\ref{RHS1}) implies (\ref{beckner}). As one can see the improvement (\ref{impr2}) is essential, for example, consider $y \to \infty$, or consider $x \to 0$.

Theorem~\ref{mth1} provides as with inequality (\ref{RHS1}) on the discrete cube $\{-1,1\}^{n}$ and for any real valued $f$ (not necessarily positive). Indeed, 
we notice that $M(x,0)=x_{+}^{3/2}$ where $x_{+}=\max\{0,x\}$ for any $x \in \mathbb{R}$. Therefore (\ref{mth2}) can be rewritten as follows 
\begin{align}
&\int_{\{-1,1\}^{n}} f^{3/2}_{+} d\mu - \left(\int_{\{-1,1\}^{n}} f d\mu \right)_{+}^{3/2} \leq \label{beckner} \\
&\int_{\{-1,1\}^{n}}\left(f^{3/2}_{+} -\frac{1}{\sqrt{2}}(2f-\sqrt{f^{2}+|\nabla f|^{2}})\sqrt{\sqrt{f^{2}+|\nabla f|^{2}}+f^{2}}\right) d\mu, \nonumber
\end{align}
where $d\mu$ is uniform counting measure on $\{-1,1\}^{n}$.  

Next since the left hand side of (\ref{impr2}) is decreasing function in $x$:
\begin{align*}
&\frac{\partial }{\partial x}\left(x_{+}^{3/2} -\frac{1}{\sqrt{2}}(2x-\sqrt{x^{2}+y^{2}})\sqrt{\sqrt{x^{2}+y^{2}}+x^{2}}\right) = \\
&\frac{3}{2}\left(x_{+}^{1/2} - \sqrt{\frac{\sqrt{x^{2}+y^{2}}+x}{2}} \right)<0,
\end{align*}
and $M(0,y)=\frac{1}{\sqrt{2}}y^{3/2}$ for $y \geq 0$  we obtain from (\ref{beckner}) that  for any $f \geq 0$ 
\begin{align}\label{concentration}
\int_{\{-1,1\}^{n}} f^{3/2} d\mu - \left(\int_{\{-1,1\}^{n}} f d\mu \right)^{3/2}\leq \frac{1}{\sqrt{2}} \int_{\{-1,1\}^{n}} |\nabla f|^{3/2}d\mu.
\end{align}
We notice that the bound (\ref{concentration}) does not follow from Beckner--Sobolev inequality (\ref{bbeckner}) even in the continuous setting (when $x \to 0$ the right hand side of (\ref{impr2})  goes to infinity). On the other hand it is clear that the discrete inequalities (\ref{concentration}) and (\ref{beckner}) are stronger then their continuous versions as one can see from the central limit theorem ($M(x,y)$ is continuous function). Thus applying central limit theorem (as in \cite{GROSS}, \cite{Bob5}) to (\ref{beckner}) we obtain that for any smooth bounded real valued $f$  (not necessarily positive) we have 
\begin{align}
&\int_{\mathbb{R}^{n}} f^{3/2}_{+} d\gamma - \left(\int_{\mathbb{R}^{n}} f d\gamma \right)_{+}^{3/2} \leq \label{ext1} \\
&\int_{\mathbb{R}^{n}}\left(f^{3/2}_{+} -\frac{1}{\sqrt{2}}(2f-\sqrt{f^{2}+|\nabla f|^{2}})\sqrt{\sqrt{f^{2}+|\nabla f|^{2}}+f^{2}}\right) d\gamma. \nonumber
\end{align}
(\ref{ext1}) extends our old  result (\ref{RHS1}) to the functions taking negative values as well.

 \subsection*{Acknowledgements} We are very grateful to Fedor Petrov. 

\end{document}